\pdfoutput=1
\documentclass[11pt,a4paper]{amsart}
\usepackage[margin=1in]{geometry}
\usepackage[utf8]{inputenc}
\usepackage[T1]{fontenc}
\usepackage{microtype}

\usepackage{amsmath}
\usepackage{amssymb}
\usepackage{mathtools}
\usepackage{tikz-cd}
\let\injlim\varinjlim
\let\projlim\varprojlim

\usepackage{enumitem}
\usepackage{booktabs}
\usepackage[bookmarksdepth=2,pdfencoding=unicode]{hyperref}
\usepackage[capitalise,noabbrev,nosort]{cleveref}

\crefname{equation}{}{}
\crefname{enumi}{}{}

\numberwithin{equation}{section}
\theoremstyle{plain}
\newtheorem{Theorem}{Theorem}
\expandafter\renewcommand\csname theTheorem\endcsname{\Alph{Theorem}}
\newtheorem{theorem}[equation]{Theorem}
\newtheorem{proposition}[equation]{Proposition}
\newtheorem{lemma}[equation]{Lemma}
\newtheorem{corollary}[equation]{Corollary}

\theoremstyle{definition}
\newtheorem{definition}[equation]{Definition}
\newtheorem{example}[equation]{Example}
\newtheorem{question}[equation]{Question}

\theoremstyle{remark}
\newtheorem{remark}[equation]{Remark}
\newtheorem{slogan}[equation]{Slogan}

\newcommand{\ZZ}{\mathbf{Z}}

\let\SS\relax
\newcommand{\SS}{\mathbf{S}}

\newcommand{\DD}{\mathbb{D}}

\newcommand{\R}{\textnormal{R}}

\newcommand{\fin}{\textnormal{fin}}

\newcommand{\cpt}{\textnormal{cpt}}

\newcommand{\st}{\textnormal{st}}

\newcommand{\Alex}{\operatorname{Alex}}

\newcommand{\Fun}{\operatorname{Fun}}

\newcommand{\Map}{\operatorname{Map}}

\newcommand{\Sh}{\operatorname{Sh}}
\newcommand{\Shv}{\operatorname{Shv}}
\newcommand{\PShv}{\operatorname{PShv}}
\newcommand{\cShv}{\operatorname{cShv}}
\newcommand{\Down}{\operatorname{Down}}
\newcommand{\Sub}{\operatorname{Sub}}
\newcommand{\Pow}{\operatorname{P}}
\newcommand{\D}{\operatorname{D}}
\newcommand{\Ord}{\operatorname{\Delta}}

\newcommand{\cofib}{\operatorname{cofib}}
\newcommand{\fib}{\operatorname{fib}}
\newcommand{\id}{\operatorname{id}}
\newcommand{\op}{\operatorname{op}}
\newcommand{\pr}{\operatorname{pr}}

\newcommand{\X}{\text{--}}
\newcommand{\unit}{\mathbf{1}}

\newcommand{\cat}[1]{\mathcal{#1}}
\newcommand{\Cat}[1]{\mathsf{#1}}

\let\autocite\cite

\title{Posets for which Verdier duality holds}
\author{Ko Aoki}
\address{
  Max Planck Institute for Mathematics,
  Vivatsgasse 7, 53111 Bonn, Germany
}
\email{aoki@mpim-bonn.mpg.de}
\date{\today}

\begin{document}

\begin{abstract}
  We discuss
  two known sheaf-cosheaf duality theorems:
  Curry's for the face posets of finite regular CW complexes
  and
  Lurie's for compact Hausdorff spaces,
  i.e., covariant Verdier duality.
  We provide a uniform formulation for them and prove
  their generalizations.
  Our version of the former works over the sphere spectrum
  and for more general finite posets,
  which we characterize in terms of the Gorenstein* condition.
  Our version of the latter
  says that the stabilization of a proper separated \(\infty\)-topos
  is rigid in the sense of Gaitsgory.
  As an application,
  for stratified topological spaces,
  we clarify the relation between these two duality equivalences.
\end{abstract}

\maketitle
\setcounter{tocdepth}{1}
\tableofcontents

\section{Introduction}\label{s-intro}

For the face poset~\(P\) of a locally finite regular CW complex,
Curry proved
in \autocite[Theorem~7.7]{Curry18} that
a canonical equivalence
\begin{equation}
  \label{e-26eeba67}
  \DD\colon\mathrm{h}{\D^{\textnormal{b}}(\Fun(P,\Cat{Vect}))}
  \longrightarrow\mathrm{h}{\D^{\textnormal{b}}(\Fun(P^{\op},\Cat{Vect}))}
\end{equation}
between triangulated categories\footnote{
  Here \(\mathrm{h}\) denotes the underlying triangulated category
  of a stable \(\infty\)-category.
} exists, where \(\Cat{Vect}\) denotes the category of vector spaces over a field.
Note that the polysimplicial case was proven before by Schneider
in \autocite[Proposition~2]{Schneider98}.
This leads to the following:

\begin{question}\label{a27f3b95f9}
  Under what conditions on~\(P\) do we have such a duality equivalence?
\end{question}

To answer this question,
we must give a definition of ``such a duality equivalence''
so that the desired equivalence is not an extra datum
but a property.

\begin{example}\label{dd32f45da1}
  Let \(P\) be the poset
  generated by the relations \(0\leq1\leq2\)
  and \(0\leq1'\leq2\) on the set \(\{0,1,1',2\}\).
  An isomorphism \(P\simeq P^{\op}\)
  gives an equivalence \(\Fun(P,\Cat{Sp})\simeq\Fun(P^{\op},\Cat{Sp})\),
  but the two isomorphisms give different equivalences.
  We can also use the suspension functor to get many more equivalences.
\end{example}

Curry's proof rules out this example,
but his argument depends
on a somewhat arbitrary choice of a dualizing complex.

To get some insights,
let us look at a similar equivalence in general topology.
Let \(X\) be a locally compact Hausdorff space.
In \autocite[Section~5.5.5]{LurieHA},
Lurie states Verdier duality
as a canonical equivalence
\begin{equation*}
  \DD\colon\Shv_{\Cat{Sp}}(X)\longrightarrow\cShv_{\Cat{Sp}}(X)
\end{equation*}
between the \(\infty\)-categories
of spectrum-valued sheaves and cosheaves.
The original construction is complicated, but as
we see in \cref{s-ch},
a simpler explanation exists:
We assume that \(X\) is compact for simplicity
and write
\(p\colon X\to{*}\) and
\(d\colon X\to X\times X\)
for the projection and the diagonal,
respectively.
Then the composites
\begin{gather}
  \label{e-e91f4124}
  \phantom,
  \Cat{Sp}
  \xrightarrow{p^*}\Shv_{\Cat{Sp}}(X)
  \xrightarrow{d_*}\Shv_{\Cat{Sp}}(X\times X)
  \simeq\Shv_{\Cat{Sp}}(X)\otimes\Shv_{\Cat{Sp}}(X),\\
  \label{e-e94d496b}
  \Shv_{\Cat{Sp}}(X)\otimes\Shv_{\Cat{Sp}}(X)
  \simeq\Shv_{\Cat{Sp}}(X\times X)
  \xrightarrow{d^*}\Shv_{\Cat{Sp}}(X)
  \xrightarrow{p_*}\Cat{Sp}
\end{gather}
constitute a duality datum for the self-duality
in~\(\Cat{Pr}_{\st}\),
the symmetric monoidal \(\infty\)-category
of presentable stable \(\infty\)-categories.\footnote{
  We prove in \cref{ss-ps} that \(\Shv_{\Cat{Sp}}(X)\) is rigid,
  which is stronger than this claim.
}
Now let us get back to our problem and
take a finite poset~\(P\).
It is known (see \cref{ss-alex}) that
when \(X\) is its Alexandroff space \(\Alex(P)\)
(see \cref{e407564b20}),
we have \(\Shv_{\Cat{Sp}}(X)\simeq\Fun(P,\Cat{Sp})\)
and \(\cShv_{\Cat{Sp}}(X)\simeq\Fun(P^{\op},\Cat{Sp})\).
Moreover, in this case,
both \cref{e-e91f4124,e-e94d496b}
are in \(\Cat{Pr}_{\st}\).
However, these do not form a duality datum unless \(P\) is discrete:

\begin{example}\label{3d8480c654}
  Assume \(X=\Alex(P)\) for a finite nondiscrete poset~\(P\).
  Pick an element \(p\in P\) that is minimal among
  nonminimal elements.
  Let \(F\colon P\to\Cat{Sp}\) be the extension
  of \(\SS\) by zero along \(\{p\}\hookrightarrow P\).
  Then the value at~\(p\) of the image of~\(F\) under
  \begin{equation*}
    \Shv_{\Cat{Sp}}(X)
    \xrightarrow{\id\otimes\text{\cref{e-e91f4124}}}
    \Shv_{\Cat{Sp}}(X)
    \otimes\Shv_{\Cat{Sp}}(X)
    \otimes\Shv_{\Cat{Sp}}(X)
    \xrightarrow{\cref{e-e94d496b}\otimes{\id}}
    \Shv_{\Cat{Sp}}(X)
  \end{equation*}
  is given by the limit \(\projlim(F)\),
  which is a coproduct of \((\#(P_{<p})-1)\)
  copies of \(\Sigma^{-1}\SS\) and
  never equivalent to~\(F(p)=\SS\).
\end{example}

Therefore we give up to use \cref{e-e91f4124}
and consider if the composite
\begin{equation}
  \label{e-6dded2c2}
  \phantom,
  \Shv_{\Cat{Sp}}(X)
  \xrightarrow{\X\otimes\Shv_{\Cat{Sp}}(X)}
  [\Shv_{\Cat{Sp}}(X),\Shv_{\Cat{Sp}}(X)\otimes\Shv_{\Cat{Sp}}(X)]
  \xrightarrow{[\id,\text{\cref{e-e94d496b}}]}
  [\Shv_{\Cat{Sp}}(X),\Cat{Sp}],
\end{equation}
where \([\X,\X]\) denotes the internal mapping object in \(\Cat{Pr}_{\st}\),
is an equivalence.
Thus we reach the following:

\begin{definition}\label{89fa6d69dd}
  We call a finite poset \(P\) \emph{Verdier}
  if the pair \((\Fun(P,\Cat{Sp}),\Gamma)\)
  is a commutative Frobenius algebra
  (cf. \cref{ss-dualizable}) in \(\Cat{Pr}_{\st}\).\footnote{
    This condition is a~priori different
    from the requirement that
    \cref{e-6dded2c2} is an equivalence,
    but it turns out to be equivalent by \cref{a3ad5064a0}
    as \(\Shv_{\Cat{Sp}}(X)\simeq\Fun(P,\Cat{Sp})\)
    is compactly generated in this case.
  }
\end{definition}

In other words, \(P\) is Verdier
if and only if
\cref{e-e94d496b} is a ``perfect pairing''.
Our main result rephrases the Verdier property
in terms of the Gorenstein* property,
a concept used in combinatorial commutative algebra:

\begin{Theorem}\label{main}
  For a finite poset~\(P\),
  the following are equivalent:
  \begin{enumerate}[label=(\roman*)]
    \item
      \label{i-verdier}
      The poset \(P\) is Verdier.
    \item
      \label{i-gorenstein}
      For each \(p\in P\),
      the full subposet \(P_{<p}\) is Gorenstein* (over~\(\ZZ\)),
      i.e., its geometric realization
      is a generalized homology sphere
      (see \cref{ss-g}).
    \item
      \label{i-vanishing}
      For each \(p<q\) in~\(P\),
      the limit of \(\ZZ_{[p,q]}\),
      the extension of the constant functor~\(\ZZ\)
      by zero along \([p,q]\hookrightarrow P\)
      vanishes.
      (Or equivalently, the limit of \(\SS_{[p,q]}\)
      vanishes; see \cref{sz}.)
  \end{enumerate}
\end{Theorem}

The existence
of an equivalence \(\Fun(P,\Cat{Sp})\simeq\Fun(P^{\op},\Cat{Sp})\)
for~\(P\) satisfying \cref{i-gorenstein}
or its variant where \(\Cat{Sp}\) is replaced by \(\D(\ZZ)\)
may not surprise experts.
What is novel is the formulation itself,
with which an if-and-only-if statement becomes possible.

\begin{corollary}\label{cc3fc211f5}
  A finite poset~\(P\) is Gorenstein* if and only if
  \(P^{\triangleright}\) is Verdier.
\end{corollary}

\begin{example}\label{3cad4d5d3f}
  As proven in \autocite[Proposition~3.1]{Bjorner84},
  a finite poset~\(P\) is the face poset
  of some regular CW complex if and only if
  the geometric realization of \(P_{<p}\) is homeomorphic to a sphere
  for each~\(p\).
  Hence any finite face poset is Verdier.
  In particular, we have an equivalence
  \(\Fun(P,\Cat{Sp})\simeq\Fun(P^{\op},\Cat{Sp})\).
\end{example}

In light of this example,
the equivalence
\(\text{\cref{i-verdier}}\Leftrightarrow\text{\cref{i-gorenstein}}\)
in \cref{main}
can be informally summarized by the following:

\begin{slogan}\label{dcf6d450ff}
  A finite poset enjoys Verdier duality
  if and only if it is homologically CW.
\end{slogan}

Of course,
there is an example that is not a face poset:

\begin{example}\label{a8f461525b}
  Let \(P\) be the face poset
  of the triangulation
  of a homology sphere
  that is not a sphere.
  Then \(P^{\triangleright}\)
  is Verdier,
  but it is not
  the face poset of
  any regular CW complex.
\end{example}

\begin{remark}\label{e14ab4a7a8}
  In this paper,
  we work over~\(\SS\) (or~\(\ZZ\)) for simplicity,
  but our argument is valid over other coefficients.
  For example,
  for a field~\(k\)
  and a finite poset~\(P\),
  the functor
  \(\projlim\colon\Fun(P,\D(k))\to\D(k)\)
  makes \(\Fun(P,\D(k))\)
  a commutative Frobenius algebra
  in the \(\infty\)-category
  of \(k\)-linear (presentable)
  stable \(\infty\)-categories
  if and only if
  \(P_{<p}\) is Gorenstein* over~\(k\)
  for any \(p\in P\).
\end{remark}

As a byproduct of our proof,
we find the following
generalization of \autocite[Proposition~1.2.4.3]{LurieHA},
which may be of independent interest.

\begin{Theorem}\label{0b4fe90159}
  Let \(P\) be a Gorenstein* finite poset.
  Then for any stable \(\infty\)-category \(\cat{C}\),
  a diagram \((P^{\triangleright})^{\triangleleft}
  \simeq(P^{\triangleleft})^{\triangleright}
  \to\cat{C}\) is
  a limit if and only if it is a colimit.
\end{Theorem}

We can handle the locally finite case
by a limit argument;
precisely, we show the following:

\begin{Theorem}\label{lf-arbitrary}
  Let \(P\) a poset such that
  \(P_{\geq p}\) is finite
  and \(P_{<p}\)
  is finite and Gorenstein* for each \(p\in P\)
  (e.g., the face poset of a locally finite regular CW complex).
  Then there is a canonical equivalence
  \begin{equation*}
    \phantom,
    \DD\colon
    \Fun(P,\Cat{Sp})
    \longrightarrow\Fun(P^{\op},\Cat{Sp}),
  \end{equation*}
  which is pointwise given by
  \(\DD(F)\colon p\mapsto\projlim_{q\in P}\Map(p,q)\otimes F(q)\),
  where \(\otimes\) denotes the copower.
\end{Theorem}

\begin{remark}\label{4c27a68603}
  In \cref{lf-arbitrary},
  the requirement that \(P_{\geq p}\) be finite
  cannot be dropped:
  Let \(P\) be the poset
  given in \autocite[Example~A.13]{ttg-fun}.
  Then \(P_{<p}\) is Gorenstein* for any \(p\);
  in fact, \(P\) is the face
  poset of a regular CW structure of~\(S^{\infty}\).
  However, the assignment
  described in the statement
  does not preserve compact objects,
  thus does not lift to an equivalence.
\end{remark}

Our formulation gives us more than aesthetic satisfaction.
For instance,
this unified view to the two duality theorems
enables us to study the interaction between startification and 
Verdier duality.
A sample application is the following:

\begin{Theorem}\label{str}
  Let \(X\to\Alex(P)\)
  be a stratification of a compact Hausdorff space,
  where \(P\) is a Verdier finite poset.
  Suppose that
  the inverse image \(\Shv_{\Cat{Sp}}(\Alex(P))\to\Shv_{\Cat{Sp}}(X)\)
  is fully faithful.
  Then our duality functor
  \(\Fun(P,\Cat{Sp})\to\Fun(P^{\op},\Cat{Sp})\)
  can be canonically identified with the composite
  \begin{equation*}
    \phantom,
    \Shv_{\Cat{Sp}}(\Alex(P))
    \xrightarrow{f^*}
    \Shv_{\Cat{Sp}}(X)
    \xrightarrow{\DD}
    \cShv_{\Cat{Sp}}(X)
    \xrightarrow{f_+}
    \cShv_{\Cat{Sp}}(\Alex(P)),
  \end{equation*}
  where \(\DD\) is the Verdier duality equivalence for~\(X\)
  and \(f_+\) is the cosheaf pushforward.
\end{Theorem}

\begin{example}\label{8fb3bcf249}
  In a nice situation,
  the (space-valued) inverse image
  \(\Shv(\Alex(P))\to\Shv(X)\) is fully faithful and its image
  consists of constructible sheaves;
  see \autocite[Section~3]{ClausenJansen}\footnote{
    Beware that this part contains
    a minor error; see \cref{c204a3567d}.
  } for a precise statement.
  For example,
  we can show that the assumption of \cref{str}
  is satisfied when
  \(X\) is a finite regular CW complex
  and \(P\) is its face poset.
\end{example}

This paper is organized as follows:
We develop necessary tools on duality in \cref{s-duality}
and on poset (co)homology in \cref{s-poset}.
Then we prove \cref{main} by showing
\(\text{\cref{i-gorenstein}}\Rightarrow\text{\cref{i-vanishing}}\),
\(\text{\cref{i-vanishing}}\Leftrightarrow\text{\cref{i-verdier}}\),
and \(\text{\cref{i-verdier}}\Rightarrow\text{\cref{i-gorenstein}}\)
in \cref{ss-van,ss-fin,ss-up}, respectively.
We also show \cref{0b4fe90159} in \cref{ss-up}.
After that, we study its variants
in \cref{s-v} and in particular prove \cref{lf-arbitrary}.
In \cref{s-ch},
we study Verdier duality for locally compact Hausdorff spaces
from a formal standpoint.
It motivates our formulation
and is used to obtain \cref{str}.

\subsection*{Conventions}

For a poset~\(P\), we write \(P_{\bot}\) and \(P_{\top}\)
for the posets obtained by adding the least element~\(\bot\)
and the greatest element~\(\top\), respectively.
When we regard \(P\) as an \(\infty\)-category,
these correspond to its left and right cones
(\(P^{\triangleleft}\) and~\(P^{\triangleright}\))
in \autocite[Notation~1.2.8.4]{LurieHTT}.
We also write
\(P_{\bot,\top}\) for the one obtained by adding both.
The empty face (or \((-1)\)-face) is not included in our face poset,
but we regard \(S^{-1}=\emptyset\) as a sphere.

We use the closed symmetric monoidal structure
on~\(\Cat{Pr}\) given in \autocite[Section~4.8.1]{LurieHA}.
For an \(\infty\)-topos~\(\cat{X}\)
and a presentable \(\infty\)-category~\(\cat{C}\),
the \(\infty\)-categories
of \(\cat{C}\)-valued sheaves \(\Shv_{\cat{C}}(\cat{X})\)
and cosheaves \(\cShv_{\cat{C}}(\cat{X})\)
are identified with
\(\cat{C}\otimes\cat{X}\) and \([\cat{X},\cat{C}]\),
respectively.
Concretely, their objects can be regarded as
limit-preserving functors \(\cat{X}^{\op}\to\cat{C}\)
and colimit-preserving functors \(\cat{X}\to\cat{C}\), respectively.
We write \(f_+\dashv f^+\) for
the pushforward-pullback adjunction for cosheaves.
The global section functor, i.e., the cohomology functor,
is denoted by~\(\Gamma\), not by \(\mathrm{R}\Gamma\).

\subsection*{Acknowledgments}

While working on this project,
I was at the University of Tokyo
and the Max Planck Institute for Mathematics
and was partially supported by the Hausdorff Center for Mathematics.
I thank them for the hospitality.

\section{General facts on duality}\label{s-duality}

\subsection{A useful criterion}\label{ss-dualizable}


Recall that
for objects \(A\), \(A^{\vee}\) and a morphism
\(e\colon A^{\vee}\otimes A\to\unit\)
in a symmetric monoidal \(\infty\)-category,
we say that \(e\) is a \emph{counit} of a duality
between \(A\) and~\(A^{\vee}\)
if for any objects \(C\) and~\(D\) the composite
\begin{equation*}
  \Map(C,D\otimes A^{\vee})
  \xrightarrow{\X\otimes A}
  \Map(C\otimes A,D\otimes A^{\vee}\otimes A)
  \xrightarrow{\Map(C\otimes A,D\otimes e)}
  \Map(C\otimes A,D)
\end{equation*}
is an equivalence.

\begin{lemma}\label{a3ad5064a0}
  Let \(A\) and \(A^{\vee}\) be objects
  and \(e\colon A^{\vee}\otimes A\to\unit\)
  a morphism
  in a closed symmetric monoidal \(\infty\)-category.
  If \(A\) is dualizable and the composite
  \begin{equation}
    \label{e-90ab6b96}
    A^{\vee}
    \simeq[\unit,A^{\vee}]
    \xrightarrow{\X\otimes A}[A,A^{\vee}\otimes A]
    \xrightarrow{[A,e]}[A,\unit]
  \end{equation}
  is an equivalence,
  then \(e\) is a counit.
  Here \([\X,\X]\) denotes the mapping object functor.
\end{lemma}

\begin{proof}
  By the definition of \([\X,\X]\), it suffices to show that
  the morphism
  \begin{equation*}
    C\otimes A^{\vee}
    \simeq[\unit,C\otimes A^{\vee}]
    \xrightarrow{\X\otimes A}[A,C\otimes A^{\vee}\otimes A]
    \xrightarrow{[A,C\otimes e]}[A,C]
  \end{equation*}
  is an equivalence for every~\(C\).
  Since \(A\) is dualizable,
  this morphism is equivalent to the one obtained by
  applying \(C\otimes\X\) to \cref{e-90ab6b96}.
\end{proof}

\subsection{Functorialities}\label{ss-pot}

For a commutative algebra \(A\)
and a morphism \(l\colon A\to\unit\)
(of objects)
in a closed symmetric monoidal \(\infty\)-category,
we can form a morphism \(A\to[A,\unit]\)
as in \cref{e-90ab6b96}
by letting \(e\) be
the composite \(l\circ m\)
where \(m\colon A\otimes A\to A\) is the multiplication.
We discuss the (\(1\)-categorical) naturality of this assignment
\((A,l)\mapsto(A\to[A,\unit])\).
Note that the pair \((A,l)\) is called
a \emph{commutative Frobenius algebra} if \(e\) is a counit.

\begin{proposition}\label{00461cedb4}
  Suppose that
  \(f\colon A\to B\) is a morphism of commutative algebras
  in a symmetric monoidal \(\infty\)-category
  and that \(g\colon B\to A\) is an \(A\)-linear morphism.
  Then for every morphism \(l\colon A\to\unit\)
  and any objects \(C\) and~\(D\),
  there are commutative squares
  \begin{align*}
    \begin{tikzcd}[ampersand replacement=\&]
      \Map(C,D\otimes A)\ar[r]\ar[d,"(D\otimes f)\circ\X"']\&
      \Map(C\otimes A,D)\ar[d,"\X\circ(C\otimes g)"]\\
      \Map(C,D\otimes B)\ar[r]\&
      \Map(C\otimes B,D)\rlap,
    \end{tikzcd}
    &&
    \begin{tikzcd}[ampersand replacement=\&]
      \Map(C,D\otimes B)\ar[r]\ar[d,"(D\otimes g)\circ\X"']\&
      \Map(C\otimes B,D)\ar[d,"\X\circ(C\otimes f)"]\\
      \Map(C,D\otimes A)\ar[r]\&
      \Map(C\otimes A,D)\rlap,
    \end{tikzcd}
  \end{align*}
  where the horizontal morphisms are
  the ones associated to
  \((A,l)\) and \((B,l\circ g)\), respectively.

  Moreover,
  if the symmetric monoidal structure is closed,
  the same thing holds when the mapping spaces
  are replaced by the internal mapping objects.
\end{proposition}

\begin{example}\label{50e6a4f108}
  Let \(K\) be an \(\infty\)-category
  and \(i\colon K_0\hookrightarrow K\) be an inclusion
  of a sieve.
  A direct computation shows that \(f
  =i^*\colon\Fun(K,\Cat{Sp})\to\Fun(K_0,\Cat{Sp})\)
  and its right adjoint \(g\)
  satisfy the assumptions
  of \cref{00461cedb4}
  in~\(\Cat{Pr}_{\st}\).
\end{example}

\begin{example}\label{b5898d2f13}
  Let \(K\) be an \(\infty\)-category
  and \(j\colon K_1\hookrightarrow K\) be an inclusion
  of a cosieve.
  A direct computation shows that
  \(f
  =j^*\colon\Fun(K,\Cat{S})\to\Fun(K_1,\Cat{S})\)
  and its left adjoint \(g\) satisfy the assumptions
  of \cref{00461cedb4}
  in~\(\Cat{Pr}\).
  This is a special case of \cref{f7ea8c5d40} below.
\end{example}

\begin{example}\label{68011c91d6}
  Let \(p\colon\cat{Y}\to\cat{X}\) be a proper geometric morphism
  between \(\infty\)-toposes.
  According to \cref{90680bda5b},
  \(f=p^*\colon\Shv_{\Cat{Sp}}(\cat{X})\to\Shv_{\Cat{Sp}}(\cat{Y})\)
  and its right adjoint~\(g\)
  satisfy the assumptions of \cref{00461cedb4}
  in~\(\Cat{Pr}_{\st}\).
\end{example}

\begin{example}\label{f7ea8c5d40}
  Let \(j\colon\cat{Y}\to\cat{X}\) be an étale geometric morphism
  between \(\infty\)-toposes.
  As noted in \autocite[Remark~6.3.5.2]{LurieHTT},
  \(f=j^*\colon\Shv(\cat{X})\to\Shv(\cat{Y})\)
  and its left adjoint~\(g\)
  satisfy the assumptions
  of \cref{00461cedb4}
  in~\(\Cat{Pr}\).
\end{example}

\begin{proof}[Proof of \cref{00461cedb4}]
  In this proof,
  in order to simplify the notation,
  we write \((\X,\X)\) for
  \(\Map(C\otimes\X,D\otimes\X)\)
  or \([C\otimes\X,D\otimes\X]\) if
  the symmetric monoidal structure is closed.


  For the first square,
  we construct the \(2\)-cells in the diagram
  \begin{equation*}
    \begin{tikzcd}
      &&(A,A\otimes A)\ar[r]\ar[d,"{(g,A\otimes A)}"]&
      (A,A)\ar[r]\ar[d,"{(g,A)}"]&
      (A,\unit)\ar[d,"{(g,\unit)}"]\\
      (\unit,A)\ar[r]\ar[d,"{(\unit,f)}"']\ar[rru]&
      (B,A\otimes B)\ar[r]\ar[d,"{(B,f\otimes B)}"']&
      (B,A\otimes A)\ar[r]&
      (B,A)\ar[r]&
      (B,\unit)\rlap.\\
      (\unit,B)\ar[r]&
      (B,B\otimes B)\ar[r]&
      (B,B)\ar[ru]
    \end{tikzcd}
  \end{equation*}
  We obtain the triangle and the four rectangles
  by naturality.
  We also obtain the pentagon
  by the linearity of~\(g\) and the functoriality of \((B,\X)\).

  For the second square,
  we construct the \(2\)-cells in the diagram
  \begin{equation*}
    \begin{tikzcd}
      &&(B,B\otimes B)\ar[r]\ar[d,"{(f,B\otimes B)}"]&
      (B,B)\ar[r]\ar[d,"{(f,B)}"]&
      (B,A)\ar[r]\ar[d,"{(f,A)}"]&
      (B,\unit)\ar[d,"{(f,\unit)}"]\\
      (\unit,B)\ar[r]\ar[d,"{(\unit,g)}"']\ar[rru]&
      (A,B\otimes A)\ar[r]\ar[d,"{(A,g\otimes A)}"']&
      (A,B\otimes B)\ar[r]&
      (A,B)\ar[r]&
      (A,A)\ar[r]&
      (A,\unit)\rlap.\\
      (\unit,A)\ar[r]&
      (A,A\otimes A)\ar[rrru]
    \end{tikzcd}
  \end{equation*}
  We obtain the upper triangle and the four rectangles
  by naturality.
  We also obtain the lower triangle by the linearity of~\(g\)
  and the functoriality of \((A,\X)\).
\end{proof}

We record the following obvious consequence:

\begin{corollary}\label{ee41f58df5}
  In the situation of \cref{00461cedb4},
  assume furthermore that \((A,l)\) is Frobenius
  and that \(f\circ g\) is homotopic to the identity.
  Then \((B,l\circ g)\) is also Frobenius.
\end{corollary}

We also have the following variant:

\begin{lemma}\label{e894284df5}
  Suppose that
  \(f\colon A\to B\) is a morphism of commutative algebras
  in a symmetric monoidal \(\infty\)-category.
  Then for every morphism \(m\colon B\to\unit\)
  and any objects \(C\) and~\(D\),
  there is a commutative square
  \begin{equation*}
    \begin{tikzcd}
      \Map(C,D\otimes A)\ar[r]\ar[d,"(D\otimes f)\circ\X"']&
      \Map(C\otimes A,D)\\
      \Map(C,D\otimes B)\ar[r]&
      \Map(C\otimes B,D)\ar[u,"\X\circ(C\otimes f)"']\rlap,
    \end{tikzcd}
  \end{equation*}
  where the horizontal morphisms are
  the ones associated to
  \((A,m\circ f)\) and \((B,m)\), respectively.

  Moreover,
  if the symmetric monoidal structure is closed,
  the same thing holds when the mapping spaces
  are replaced by the internal mapping objects.
\end{lemma}

\begin{proof}
  We use the same notation
  as in the proof of \cref{00461cedb4}.
  We construct the \(2\)-cells in the diagram
  \begin{equation*}
    \begin{tikzcd}
      (\unit,A)\ar[r]\ar[d,"{(\unit,f)}"']&
      (A,A\otimes A)\ar[r]\ar[d,"{(A,f\otimes f)}"']&
      (A,A)\ar[d,"{(A,f)}"']\ar[rd]&
      {}\\
      (\unit,B)\ar[rd]&
      (A,B\otimes B)\ar[r]&
      (A,B)\ar[r]&
      (A,\unit)\\
      {}&
      (B,B\otimes B)\ar[u,"{(f,B\otimes B)}"']\ar[r]&
      (B,B)\ar[r]\ar[u,"{(f,B)}"']&
      (B,\unit)\rlap.\ar[u,"{(f,\unit)}"']
    \end{tikzcd}
  \end{equation*}
  We obtain the upper square since
  \(f\) is a morphism of commutative algebras.
  We also obtain the other cells by naturality.
\end{proof}

\section{Homotopy theory of posets}\label{s-poset}

\subsection{Poset cohomology}\label{ss-h-poset}

\begin{definition}\label{69fc302def}
  For a poset~\(P\),
  we write \(\lvert P\rvert\) for
  the geometric realization
  (as a topological space) of its nerve
  and \(\Ord(P)\)
  for its \emph{order complex}, i.e.,
  the abstract simplicial complex
  consisting of finite (nonempty) chains in~\(P\).
  Note that \(\lvert P\rvert\)
  is canonically homeomorphic to
  the geometric realization of \(\Ord(P)\).
\end{definition}

In this subsection,
we study how
the cohomology of~\(\lvert P\rvert\)
and that of~\(P\),
i.e., the sheaf cohomology of \(\Fun(P,\Cat{S})\),
are related.
We first recall the following from
\autocite[Section~A.1]{LurieHA}:

\begin{definition}\label{6c3f6be635}
  We say that an \(\infty\)-topos \(\cat{X}\)
  has \emph{constant shape}
  if the shape \(\Sh\cat{X}\)
  is corepresentable.
  If \(\cat{X}_{/X}\) has constant shape
  for every \(X\in\cat{X}\),
  we say that \(\cat{X}\) is \emph{locally of constant shape}.
  According to \autocite[Proposition~A.1.18]{LurieHA},
  this is equivalent to the condition that
  the constant sheaf functor \(\cat{S}\to\cat{X}\)
  admits a left adjoint.
\end{definition}

\begin{example}\label{940a2e3d8c}
  The presheaf \(\infty\)-topos
  of an \(\infty\)-category
  is locally of constant shape.
  Its shape is the image
  under the left adjoint of
  \(\Cat{S}\hookrightarrow\Cat{Cat}_{\infty}\).
\end{example}

\begin{example}\label{df7429c55a}
  The sheaf \(\infty\)-topos
  of a CW complex is locally of constant shape.
  Its shape is the homotopy type.
  In fact, any CW complex
  is locally of singular shape in the sense 
  of \autocite[Section~A.4]{LurieHA}
  as any open subspace
  is homotopy equivalent to a CW complex.
\end{example}

\begin{proposition}\label{cb58ae3d6a}
  If an \(\infty\)-topos \(\cat{X}\)
  is locally of constant shape,
  for any spectrum~\(E\),
  the canonical morphism
  \([\Sigma^{\infty}_+\Sh\cat{X},E]
  \to\Gamma(\cat{X};E)\)
  is an equivalence.
  Here \([\X,\X]\) denotes the mapping spectrum.
\end{proposition}

\begin{proof}
  Let \(p\colon\cat{X}\to\Cat{S}\) denote the projection.
  By assumption,
  \(p^*\) admits a left adjoint~\(p_!\).
  If we regard objects
  in \(\Shv_{\Cat{Sp}}(\X)\)
  as limit-preserving functors
  \((\X)^{\op}\to\Cat{Sp}\),
  the spectrum-valued pullback
  \(\Cat{Sp}\to\Shv_{\Cat{Sp}}(\cat{X})\)
  is given as the precomposition with~\((p_!)^{\op}\).
  Therefore,
  \(\Gamma(\cat{X};E)\simeq
  p_*p^*E\) is given
  the value of \(E\colon\Cat{S}^{\op}\to\Cat{Sp}\)
  at \(p_!p^*{*}\simeq\Sh\cat{X}\),
  which is
  the cohomology \([\Sigma^{\infty}_+\Sh\cat{X},E]\).
\end{proof}

\begin{corollary}\label{afc4ccfff7}
  For a poset~\(P\) and a spectrum~\(E\),
  we have a functorial (both in~\(P\) and in~\(E\))
  equivalence
  \(\Gamma(P;E)\simeq\Gamma(\lvert P\rvert;E)\),
  where \(E\) denotes the constant sheaves
  on the \(\infty\)-toposes
  \(\Fun(P,\Cat{S})\) and \(\Shv(\lvert P\rvert)\),
  respectively.
\end{corollary}

\begin{proof}
  This follows
  from
  \cref{940a2e3d8c,df7429c55a,cb58ae3d6a}.
\end{proof}

\begin{remark}\label{f2eee8dfff}
  In fact, at least if \(P_{\geq p}\)
  is finite for \(p\in P\),
  we can construct a canonical
  geometric morphism \(\Shv(\lvert P\rvert)\to\Fun(P,\Cat{S})\)
  whose inverse image functor is fully faithful.
  This shows that
  we can take any functor \(E\colon P\to\Cat{Sp}\) as a coefficient
  in the statement of \cref{afc4ccfff7},
  but we do not need this generality
  in this paper.
\end{remark}

\subsection{Gorenstein* posets}\label{ss-g}


We first recall the following notion
from combinatorial commutative algebra.
See \autocite[Chapter~II]{Stanley96} for a textbook account,
which in particular explains where the name comes from.

\begin{definition}\label{fe7384d8c5}
  We call
  an \(n\)-dimensional\footnote{Here \(n\) can be \(-1\),
  so that the empty complex is Gorenstein*.}
  finite abstract simplicial complex
  \emph{Gorenstein*}
  if its geometric realization
  is a generalized homology \(n\)-sphere,
  i.e., an (integral) homology \(n\)-manifold
  having the (integral) homology of an \(n\)-sphere.
\end{definition}

The following definition is
a variant of the definition of a Cohen--Macaulay poset
given in \autocite[Section~3]{Baclawski80}.

\begin{definition}\label{8e20e56844}
  We call
  a finite poset~\(P\)
  \emph{Gorenstein*} if 
  for every \(p<q\) in \(P_{\bot,\top}\)
  the interval \((p,q)\) has the (integral) homology of a sphere\footnote{
    We regard \(S^{-1}=\emptyset\) as a sphere.
  }.
\end{definition}

By definition if \(P\) is a Gorenstein* finite poset
then \((p,q)\) is Gorenstein* for every \(p<q\in P_{\bot,\top}\).

\begin{lemma}\label{37f851549b}
  Any maximal chain of
  a Gorenstein* finite poset~\(P\) has the same length.
  In other words, \(P_{\bot,\top}\) admits a rank function\footnote{
    A \emph{rank function} on a finite poset~\(P\) is
    a function \(r\colon P\to\ZZ\)
    such that \(r(q)=r(p)+1\) if \(q\) is an immediate successor of~\(p\).
  }.
\end{lemma}

\begin{proof}
  This holds more generally
  for Cohen--Macaulay finite posets;
  see \autocite[Proposition~3.1]{Baclawski80}.
\end{proof}

These two definitions are compatible:

\begin{proposition}\label{35d43b2aa2}
  For a finite poset~\(P\),
  it is Gorenstein* if and only if
  \(\Ord(P)\) is Gorenstein*.
\end{proposition}

We omit the proof since
it is a straightforward variant of
\autocite[Proposition~3.3]{Baclawski80}.

\begin{corollary}\label{594984da1c}
  For a finite abstract simplicial complex,
  it is Gorenstein* if its underlying poset is Gorenstein*.
\end{corollary}

We later need the following lemma, as we prefer cohomology:

\begin{lemma}\label{94f6eb4680}
  For a finite poset,
  the Gorenstein* condition can
  be checked via cohomology instead of homology;
  i.e., 
  \(P\) is Gorenstein* if and only if
  \((p,q)\) has
  the cohomology of a sphere
  for \(p<q\) in \(P_{\bot,\top}\).
\end{lemma}

\begin{proof}
  By definition
  \(P\) is Gorenstein*
  if and only if so is \(P^{\op}\).
  Hence the desired result
  follows from the self-duality
  of the \(\infty\)-category
  of perfect complexes over~\(\ZZ\).
\end{proof}

\subsection{A vanishing result}\label{ss-van}


\begin{definition}\label{96f462d60f}
  Let \(P\) be a poset and \(E\) a spectrum.
  For \(p\leq q\) in~\(P\),
  we let \(E_{[p,q]}\in\Fun(P,\Cat{Sp})\)
  denote the functor obtained from
  the constant functor \(E\in\Fun([p,q],\Cat{Sp})\)
  by left Kan extending along
  \([p,q]\hookrightarrow P_{\leq q}\)
  and then right Kan extending along
  \(P_{\leq q}\hookrightarrow P\).
  If \(E\in\D(\ZZ)\),
  we use the same symbol
  for the element in \(\Fun(P,\D(\ZZ))\)
  determined similarly.
\end{definition}

\begin{proposition}\label{271a46a032}
  For a Gorenstein* finite poset~\(P\),
  for every \(p<q\) in \(P_{\top}\)
  the cohomology \(\Gamma(P_{\top};\ZZ_{[p,q]})\) vanishes.
\end{proposition}

We later prove the converse.

\begin{proof}
  Since \(\Gamma(P_{\top};\ZZ_{[p,q]})
  \simeq\Gamma((P_{\top})_{\leq q};\ZZ_{[p,q]})\) holds
  and \((P_{\top})_{<q}\) is also Gorenstein*,
  we can assume \(q=\top\).
  By \cref{35d43b2aa2}, there is a unique rank function
  \(r\colon P_{\bot,\top}\to\ZZ\) satisfying \(r(\bot)=-1\).
  Then \(\lvert P\rvert\)
  is a generalized homology \((r(\top)-1)\)-sphere.
  If \(r(\top)=1\) holds,
  \(P\) is the discrete poset with two elements
  and the result can be directly checked.
  So we henceforth assume \(r(\top)>1\).

  In what follows,
  we repeatedly use \cref{afc4ccfff7}.
  Let \(\ZZ_{\geq p}\)
  denote the left Kan extension
  of the constant functor
  with value \(\ZZ\) along \(P_{\geq p}\hookrightarrow P\).
  Then the pullback diagram
  \begin{equation*}
    \begin{tikzcd}
      \Gamma(P_{\top};\ZZ_{[p,\top]})\ar[r]\ar[d]&
      \Gamma(P_{\top};\ZZ)\ar[d,"f"]\\
      \Gamma(P;\ZZ_{\geq p})\ar[r,"g"]&
      \Gamma(P;\ZZ)
    \end{tikzcd}
  \end{equation*}
  can be formed in \(\D(\ZZ)\).
  Here \(f\) is induced
  by \(P\to P_{\top}\).
  As this morphism of posets induces an isomorphism
  on \(H^0(\lvert\X\rvert;\ZZ)\),
  we see that \(f\) induces an isomorphism on~\(\pi_0\).
  On the other hand,
  \(\Gamma(P;\ZZ_{\geq p})\)
  is computed as the relative cohomology
  \(\fib(\Gamma(P;\ZZ)\to\Gamma(P\setminus P_{\geq p};\ZZ))\).
  By the Lefschetz duality theorem,
  it is the dual of \(\Sigma^{r(\top)-1}\Gamma(P_{\geq p};\ZZ)\)
  in \(\D(\ZZ)\),
  which is \(\Sigma^{1-r(\top)}\ZZ\),
  and \(g\) induces an isomorphism on~\(\pi_{1-r(\top)}\)
  as \(\lvert P_{\geq p}\rvert\) is connected.
  Therefore,
  \(f\) and~\(g\)
  can be identified with
  the two direct summand inclusions
  of \(\Gamma(P;\ZZ)\simeq\ZZ\oplus\Sigma^{1-r(\top)}\ZZ\),
  from which \(\Gamma(P_{\top};\ZZ_{[p,\top]})\simeq0\) follows.
\end{proof}

\begin{proof}[Proof of
  \(\text{\cref{i-gorenstein}}\Rightarrow\text{\cref{i-vanishing}}\)
  of \cref{main}]
  This follows from \cref{271a46a032}.
\end{proof}

We note that
this vanishing also holds
when \(\ZZ\) is replaced by~\(\SS\):

\begin{lemma}\label{sz}
  For a finite poset~\(P\)
  and \(p\leq q\) in \(P\),
  the cohomology \(\Gamma(P;\ZZ_{[p,q]})\)
  vanishes if and only if
  so does \(\Gamma(P;\SS_{[p,q]})\).
\end{lemma}

\begin{proof}
  Note that
  if a spectrum \(E\) is nonzero and bounded below,
  \(E\otimes\ZZ\) is also nonzero;
  this can be seen by considering
  the smallest \(i\) such that \(\pi_iE\) is nonzero.
  Hence the desired result follows from
  \(\Gamma(P;\ZZ_{[p,q]})\simeq\Gamma(P;\SS_{[p,q]})\otimes\ZZ\)
  and the fact that \(\Gamma(P;\SS_{[p,q]})\) is bounded below,
  both of which follow from the finiteness of~\(P\).
\end{proof}

\section{Verdier duality for finite posets}\label{s-main}

\subsection{Recollements}\label{ss-recolle}

We refer the reader to \autocite[Section~A.8]{LurieHA}
for a discussion on recollements using \(\infty\)-categories.
When we say \(\cat{C}_0\) and \(\cat{C}_1\) form
a recollement, \(\cat{C}_0\) is supposed to be the ``closed'' part;
i.e., the \(\cat{C}_1\)-localization annihilates~\(\cat{C}_0\).
We abuse terminology
to say the two functors \(\cat{C}_0\hookrightarrow\cat{C}\) and
\(\cat{C}_1\hookrightarrow\cat{C}\) determine a recollement
when \(\cat{C}\) is a recollement of their images.

We recall the following standard fact:

\begin{lemma}\label{a9ca3865c3}
  Consider a presentable stable \(\infty\)-category \(\cat{C}\)
  and suppose that
  \(j\colon P_1\hookrightarrow P\) be an upward-closed full subposet
  with complement \(i\colon P_0\hookrightarrow P\).
  Let \(i_*\) and~\(j_*\) denote the right Kan extension along~\(i\) and~\(j\)
  and \(j_!\) the left Kan extension along~\(j\).
  Then the following hold for the functor \(\infty\)-category \(\Fun(P,\cat{C})\):
  \begin{enumerate}
    \item
      \label{i-sheaf}
      The functors \(i_*\) and \(j_*\) form a recollement.
    \item
      \label{i-cosheaf}
      The functors \(j_!\) and \(i_*\) also form a recollement.
  \end{enumerate}
\end{lemma}

\begin{proof}
  Both can be easily checked
  by using
  \autocite[Proposition~A.8.20]{LurieHA}
  and observing that
  \(i_*\) and \(j_!\) are
  given as the extension-by-zero functors.
\end{proof}

\begin{lemma}\label{2953dbadf2}
  Consider a left exact functor \(\cat{C}\to\cat{C}'\)
  and suppose that
  \(\cat{C}\) and \(\cat{C}'\) are recollements of
  \(\cat{C}_0\) and~\(\cat{C}_1\)
  and
  \(\cat{C}_0'\) and~\(\cat{C}_1'\), respectively.
  Furthermore, assume the following:
  \begin{itemize}
    \item
      The functor \(f\) restricts to define
      equivalences \(\cat{C}_0\to\cat{C}_1\)
      and \(\cat{C}_0'\to\cat{C}_1'\).
    \item
      The morphism \(L_0'\circ f\to f\circ L_0\)
      obtained by the above condition is an equivalence.
  \end{itemize}
  Then \(f\) itself is an equivalence.
\end{lemma}

\begin{proof}
  This follows from \autocite[Proposition~A.8.14]{LurieHA}.
\end{proof}

\subsection{Duality and the vanishing condition}\label{ss-fin}

We prove 
\cref{i-verdier,i-vanishing} in \cref{main} are equivalent.
We start with a pointwise description of~\(\DD\):

\begin{lemma}\label{f6072744fd}
  Let \(K\) be a finite \(\infty\)-category
  so that \(\projlim\colon\Fun(K,\Cat{Sp})\to\Cat{Sp}\) is
  a morphism in \(\Cat{Pr}_{\st}\).
  Then the potential duality functor
  \begin{equation*}
    \DD\colon
    \Fun(K,\Cat{Sp})\longrightarrow
    [\Fun(K,\Cat{Sp}),\Cat{Sp}]\simeq\Fun(K^{\op},\Cat{Sp})
  \end{equation*}
  induced by the composite
  \(
  \Fun(K,\Cat{Sp})
  \otimes\Fun(K,\Cat{Sp})
  \xrightarrow{\X\otimes\X}
  \Fun(K,\Cat{Sp})
  \xrightarrow{\projlim}\Cat{Sp}
  \) (cf.~\cref{e-90ab6b96})
  is objectwise given by
  \begin{equation}
    \label{e-8b7f0497}
    \phantom,
    F\longmapsto\biggl(
      k\longmapsto\projlim_{l\in K}\Map(k,l)\otimes F(l)
    \biggr),
  \end{equation}
  where \(\otimes\) denotes the copower.
\end{lemma}

\begin{proof}
  By definition,
  \(\DD\) is the composite
  \(\Fun(K,\Cat{Sp})\to
  \Fun(K^{\op}\times K\times K,\Cat{Sp})\to
  \Fun(K^{\op},\Cat{Sp})\),
  where the first and second maps are objectwise given by
  \(F\mapsto((k,l,m)\mapsto\Map(k,l)\otimes F(m))\)
  and
  \(G\mapsto\bigl(k\mapsto\projlim_lG(k,l,l)\bigr)\), respectively.
\end{proof}

\begin{proof}[Proof of
  \(\text{\cref{i-vanishing}}\Rightarrow\text{\cref{i-verdier}}\)
  of \cref{main}]
  We show by induction on~\(\#P\)
  that \(\DD_P\) is an equivalence,
  which is equivalent to the Verdier property
  by \cref{a3ad5064a0}.
  If \(P=\emptyset\), the claim is obvious.
  We assume \(\#P>0\) and pick a maximal element
  \(m\in P\). Since \(\{m\}\) is upward closed,
  by applying \cref{a9ca3865c3}
  to \(j\colon\{m\}\hookrightarrow P\)
  and \(i\colon P\setminus\{m\}\hookrightarrow P\),
  we can form two recollements, which fit into a diagram
  \begin{equation}
    \label{e-a7e63710}
    \begin{tikzcd}
      \Fun(P\setminus\{m\},\Cat{Sp})\ar[r,"i_*",hook]\ar[d,"\DD_{P\setminus\{m\}}"']&
      \Fun(P,\Cat{Sp})\ar[d,"\DD_{P}"']&
      \Cat{Sp}\ar[l,"j_*"',hook']\ar[d,dashed]\\
      \Fun((P\setminus\{m\})^{\op},\Cat{Sp})\ar[r,"i_+",hook]&
      \Fun(P^{\op},\Cat{Sp})&
      \Cat{Sp}\rlap,\ar[l,"j_?"',hook']
    \end{tikzcd}
  \end{equation}
  where the identification
  \(\Fun(\{m\},\Cat{Sp})
  \simeq\Cat{Sp}\simeq\Fun(\{m\}^{\op},\Cat{Sp})\)
  is made and
  \(j_?\) denotes the right adjoint of~\(j^+\).
  \Cref{00461cedb4}
  applied to \cref{50e6a4f108} says
  that the left square commutes
  and
  that the canonical
  morphism \(i^?\circ\DD_P\to\DD_{P\setminus\{m\}}\circ i^*\)
  is an equivalence,
  where \(i^?\) denotes the left adjoint of~\(i_+\).

  We then consider applying \cref{2953dbadf2}
  to conclude the proof.
  Since
  \(\DD_{P\setminus\{m\}}\) is an equivalence
  by our inductive hypothesis,
  it remains to check
  that \(\DD_P\)
  restricts to
  define the dashed arrow
  and that it is an equivalence.
  Equivalently, we need to show that the composite
  \begin{equation*}
    \Cat{Sp}
    \xrightarrow{j_*}\Fun(P,\Cat{Sp})
    \xrightarrow{\DD_P}\Fun(P^{\op},\Cat{Sp})
    \xrightarrow{\text{restriction}}\Fun(\{p\}^{\op},\Cat{Sp})
    \simeq\Cat{Sp}
  \end{equation*}
  is zero for \(p\neq m\) and an equivalence for \(p=m\).
  Note that since this functor is colimit-preserving,
  it is determined by its value at~\(\SS\),
  for which we write \(E_p\).
  \Cref{f6072744fd} says
  that \(E_p\)
  is computed as \(\projlim_{q\in P}\Map(p,q)\otimes(j_*\SS)(q)\).

  If \(p\nleq m\),
  the spectrum \(E_p\)
  is zero as \((j_*\SS)(q)=0\) holds for \(q\nleq m\).
  If \(p\leq m\),
  the spectrum \(E_p\) is equivalent to
  the cohomology \(\Gamma(P;\SS_{[p,m]})\).
  Hence \(E_p\) is zero for \(p<m\)
  by the assumption~\cref{i-vanishing} and \cref{sz}.
  Therefore, it remains to compute \(E_m\simeq\Gamma(P;\SS_{[m,m]})\).
  We pick a maximal chain \(p_0<\dotsb<p_r=m\) in \(P\).
  For \(i=1\), \dots,~\(n\),
  by using \(\Gamma(P;\SS_{[p_{i-1},p_{i}]})=0\),
  we have
  \begin{equation*}
    \phantom.
    \Gamma(P;\SS_{[p_i,p_i]})
    \simeq\fib\bigl(\Gamma(P;\SS_{[p_{i-1},p_i]})
    \to\Gamma(P;\SS_{[p_{i-1},p_{i-1}]})\bigr)
    \simeq\Sigma^{-1}\Gamma(P;\SS_{[p_{i-1},p_{i-1}]}).
  \end{equation*}
  Thus we have \(\Gamma(P;\SS_{[m,m]})=\Sigma^{-r}\SS\).
  Hence the dashed arrow in \cref{e-a7e63710} is identified with
  the functor \(\Sigma^{-r}\), which is an equivalence.
\end{proof}

\begin{proof}[Proof of
  \(\text{\cref{i-verdier}}\Rightarrow\text{\cref{i-vanishing}}\)
  of \cref{main}]
  We proceed by induction on~\(\#P\).
  If \(P=\emptyset\), the claim is obvious.
  We assume \(P\neq\emptyset\)
  and pick a maximal element~\(m\).
  Then \(P\setminus\{m\}\) is also Verdier
  by \cref{ee41f58df5} applied to \cref{50e6a4f108}.
  Hence by our inductive hypothesis,
  it suffices to show that
  \(\Gamma(P;\SS_{[p,m]})\) vanishes
  for any \(p<m\).

  We now form the diagram~\cref{e-a7e63710},
  but the dashed arrow already exists
  in this case since \(\DD_P\) is an equivalence.
  As we have observed in the above proof,
  the existence of the dashed arrow
  in particular means
  that \(\Gamma(P;\SS_{[p,m]})\) vanishes
  for \(p<m\), which is what we wanted to show
  by \cref{sz}.
\end{proof}

\subsection{The Gorenstein* condition from duality}\label{ss-up}

We finally complete the proof of \cref{main}.
The main ingredient is the following nontrivial observation:

\begin{proposition}\label{b644f9030c}
  If a finite poset~\(P\) is Verdier,
  \(P_{>p}\) is also Verdier for any \(p\in P\).
\end{proposition}

The proof requires the following trivial observations:

\begin{lemma}\label{6e50e2411f}
  Let \(P\) be a finite poset.
  For \(p\in P\),
  let \(\SS_{\leq p}\)
  denote the right Kan extension
  of the constant functor with value~\(\SS\)
  along \(P_{\leq p}\hookrightarrow P\).
  Then \(\Fun(P,\Cat{Sp})\) is generated
  by \(\SS_{\leq p}\) under colimits and shifts.
\end{lemma}

\begin{proof}
  We proceed by induction on \(\#P\).
  There is nothing to prove if \(P=\emptyset\).
  Assume otherwise and pick a maximal element \(m\in P\).
  Let \(\cat{C}\subset\Fun(P,\Cat{Sp})\)
  be the full subcategory generated
  by \(\SS_{\leq p}\) under colimits and shifts.
  The inductive hypothesis implies that
  \(F\in\Fun(P,\Cat{Sp})\) is in \(\cat{C}\)
  if \(F(m)\) is zero.
  Hence any \(F\in\Fun(P,\Cat{Sp})\)
  the fiber of \(F\to F(m)\otimes\SS_{\leq m}\)
  is in \(\cat{C}\).
  Since \(F(m)\otimes\SS_{\leq m}\)
  is also in \(\cat{C}\) by assumption
  we have \(F\in\cat{C}\).
\end{proof}

\begin{lemma}\label{b334c9cd53}
  Let \(P\) be a (not necessarily finite) poset
  Then \(E_{[\bot,p]}
  \colon P_{\bot}\to\Cat{Sp}\) is a limit diagram
  for any \(p\in P\)
  and any \(E\in\Cat{Sp}\).
\end{lemma}

\begin{proof}
  We can assume that \(p\) is the greatest element
  by replacing~\(P\) with \(P_{\leq p}\).
  Then the result follows from \cref{afc4ccfff7}
  or, more directly,
  the observation that now \(P\) is weakly contractible.
\end{proof}

\begin{proof}[Proof of \cref{b644f9030c}]
  By induction, it suffices to consider the case where
  \(p\) is minimal.

  Let \(j\) denote the inclusion \(P_{>p}\hookrightarrow P\)
  and \(j_!\colon\Fun(P_{>p},\Cat{Sp})
  \to\Fun(P,\Cat{Sp})\)
  the left Kan extension functor.
  \cref{ee41f58df5} applied to \cref{b5898d2f13}
  says that the pair \((\Fun(P_{>p},\Cat{Sp}),
  \Gamma_P\circ j_!)\) is a commutative Frobenius algebra.
  Hence it suffices to construct an equivalence
  \(\Gamma_P\circ j_!\simeq
  \Sigma^{-1}\circ\Gamma_{P_{>p}}\) in \(\Fun(\Fun(P_{>p},\Cat{Sp}),\Cat{Sp})\).

  We write \(j\) as the composite
  \(P_{>p}\xhookrightarrow{k}P_{\geq p}\xhookrightarrow{l}P\).
  Then we have a morphism
  \(j_!\simeq l_!\circ k_!\to l_!\circ k_*\),
  where \((\X)_!\) and \((\X)_*\) denote
  the left and right Kan extension functors.
  By applying \(\Gamma_P\circ\X\),
  we obtain \(\Gamma_P\circ j_!\to\Gamma_P\circ l_!\circ k_*\).
  Since its cofiber can be computed as
  \begin{equation*}
    \phantom,
    \Gamma_P\circ l_!\circ\cofib(k_!\to k_*)
    \simeq
    \Gamma_P\circ l_*\circ\cofib(k_!\to k_*)
    \simeq
    \Gamma_{P_{\geq p}}\circ\cofib(k_!\to k_*)
    \simeq
    \Gamma_{P_{>p}}
    ,
  \end{equation*}
  where we use the minimality of~\(p\),
  we are reduced to showing
  that \(\Gamma_P\circ l_!\circ k_*\)
  is zero.

  Let \(\cat{C}\) denote the full subcategory
  of \(\Fun(P_{\geq p},\Cat{Sp})\)
  spanned by the limit diagrams.
  We need to show that
  \(\Gamma_P\circ l_!\) is zero on~\(\cat{C}\).
  We now observe that \(\cat{C}\)
  is generated under colimits and shifts by
  \(\SS_{[p,q]}\) for \(q\in P_{>p}\):
  First, they are indeed limit diagrams by \cref{b334c9cd53}.
  Then it follows from \cref{6e50e2411f}
  that \(\Fun(P_{>p},\Cat{Sp})\) is generated by
  their restrictions.
  Therefore, we need to show that \((\Gamma_P\circ l_!)(\SS_{[p,q]})
  \simeq\Gamma_P(\SS_{[p,q]})\) is zero, but
  this follows from~\cref{i-vanishing}
  of \cref{main};
  note that we have already proven
  \(\text{\cref{i-verdier}}\Rightarrow\text{\cref{i-vanishing}}\)
  in \cref{ss-fin}.
\end{proof}

\begin{proof}[Proof of
  \(\text{\cref{i-verdier}}\Rightarrow\text{\cref{i-gorenstein}}\)
  of \cref{main}]
  Let \(P\) be a Verdier finite poset.
  According to \cref{94f6eb4680},
  it suffices to show that \((p,q)\) has
  the (integral)
  \emph{cohomology} of a sphere for \(p<q\) in \(P_{\bot}\).
  Since we know that \(P_{\leq q}\) is Verdier from
  \(\text{\cref{i-verdier}}\Leftrightarrow\text{\cref{i-vanishing}}\),
  we can assume that \(q\) is the greatest element of~\(P\).
  We can also assume that \(p=\bot\) by \cref{b644f9030c}.
  Hence it remains to compute the cohomology of~\(P_{<q}\),
  which is the fiber of
  \(\Gamma(P;\ZZ)\to\Gamma(P;\ZZ_{[q,q]})\).
  If \(P\) is a singleton, it is obviously zero.
  Otherwise, \cref{b334c9cd53}
  says \(\Gamma(P;\ZZ)\simeq\ZZ\)
  and the last part of the proof of
  \(\text{\cref{i-vanishing}}\Rightarrow\text{\cref{i-verdier}}\)
  says that \(\ZZ_{[q,q]}\)
  is some positive desuspension of~\(\ZZ\).
  Therefore, \(P_{<q}\) has the cohomology of a sphere.
\end{proof}

We then obtain \cref{0b4fe90159} as a bonus:

\begin{proof}[Proof of \cref{0b4fe90159}]
  By the standard (stable) Yoneda argument,
  we can assume \(\cat{C}=\Cat{Sp}\).
  Since \(P^{\op}\) is also Gorenstein*,
  it suffices to show that
  any limiting diagram \(P_{\bot,\top}\to\Cat{Sp}\)
  is colimiting.
  Then as in the proof of \cref{b644f9030c},
  it suffices to show
  that \(\SS_{[\bot,p]}\in\Fun(P_{\bot,\top},\Cat{Sp})\) is colimiting
  for \(p\in P_{\top}\).
  Since this is trivial for \(p=\top\)
  as \(P_{\bot}\) is weakly contractible, we assume otherwise.
  By the self-duality of
  the \(\infty\)-category of finite spectra,
  we are reduced
  to showing that \(\SS_{[p,\bot]}\in\Fun((P_{\bot,\top})^{\op},\Cat{Sp})\)
  is limiting. As \(p\neq\top\), this
  is equivalent to
  the vanishing of the cohomology of
  \(\SS_{[p,\bot]}\in\Fun((P_{\bot})^{\op},\Cat{Sp})\),
  which follows from \cref{main}.
\end{proof}

\section{Variants}\label{s-v}

In this section, we prove
\cref{lf-arbitrary}
and show that our duality can be regarded
as a topological sheaf-cosheaf duality.

\subsection{For locally finite posets}\label{ss-lf}

The equivalence \cref{e-26eeba67}
exists for the face poset
of a locally finite regular CW complex.
We extend our duality to cover that case.

\begin{definition}\label{8771652871}
  We say that a poset~\(P\) is \emph{locally finite} if
  \(P_{\geq p}\) is finite for every \(p\in P\).
\end{definition}

This terminology is justified by considering
the Alexandroff topology of~\(P\) (see \cref{e407564b20}).

\begin{definition}\label{3908ebac89}
  For a poset~\(P\),
  we write \(\Pow^{\fin}(P)\) for the poset
  of finite subsets.
  We write \(\Down(P)\) for the poset of
  sieves, i.e., downward-closed full subposets.
  We consider the functor
  \begin{equation}
    \label{e-8591975d}
    \Pow^{\fin}(P)^{\triangleright}\longrightarrow
    \Down(P)
  \end{equation}
  given by \(S\mapsto
  \bigcup_{p\in S}P_{\leq s}\)
  and \({\infty}\mapsto P\).
\end{definition}

We prove that for a nice poset~\(P\),
the presheaf \(\infty\)-categories of it and its opposite
can be recovered
from those of full subposets
of the form \(\bigcup_{p\in S}P_{\leq s}\)
for finite~\(S\)
by taking colimits in~\(\Cat{Pr}\).

\begin{proposition}\label{8d61256fe2}
  For any poset~\(P\)
  and any presentable \(\infty\)-category~\(\cat{C})\),
  the diagram given by the composite
  \begin{equation*}
    \Pow^{\fin}(P)^{\triangleright}
    \xrightarrow{\text{\cref{e-8591975d}}}
    \Down(P)
    \xrightarrow{(\PShv_{\cat{C}}(\X),\X_!)}
    \Cat{Pr}
  \end{equation*}
  is colimiting.
\end{proposition}

\begin{proof}
  First,
  note that
  the diagram
  \(\Pow^{\fin}(P)^{\triangleright}\to\Down(P)\to\Cat{Poset}\)
  is colimiting.
  Since \(\Down^{\fin}(P)\) is filtered,
  its composite with \(\Cat{Poset}\hookrightarrow\Cat{Cat}_{\infty}\)
  is also colimiting, from which the result follows.
\end{proof}

\begin{proposition}\label{13a6000b87}
  Suppose that
  \(P\) is a locally finite poset
  and \(\cat{C}\) is a compactly generated pointed \(\infty\)-category.
  Then the diagram given by the composite
  \begin{equation*}
    \Pow^{\fin}(P)^{\triangleright}
    \xrightarrow{\text{\cref{e-8591975d}}}
    \Down(P)
    \xrightarrow{(\Fun(\X,\cat{C}),\X_*)}
    \Cat{Pr}
  \end{equation*}
  is colimiting.
  Here the second arrow is well defined by \cref{0ec8cf3aa6} below.
\end{proposition}

The proof requires several lemmas:

\begin{lemma}\label{cbb999bf61}
  For a poset~\(P\),
  let \(\Down^{\fin}(P)\) be the image of
  \(\Pow^{\fin}(P)\) under \cref{e-8591975d}.
  Then \(\Pow^{\fin}(P)\to\Down^{\fin}(P)\) is cofinal.
\end{lemma}

\begin{proof}
  This follows from Joyal's version of Quillen's theorem~A
  and the fact that a nonempty poset having binary joins
  is weakly contractible.
\end{proof}

\begin{lemma}\label{0ec8cf3aa6}
  Let \(i\colon K_0\hookrightarrow K\) be a sieve inclusion
  of \(\infty\)-categories
  and \(\cat{C}\) a presentable \(\infty\)-category.
  Then the right Kan extension functor
  \(i_*\colon\Fun(K_0,\cat{C})\hookrightarrow\Fun(K,\cat{C})\) preserves
  weakly contractible colimits.
  In particular,
  \(i_*\) preserves colimits
  if \(\cat{C}\) is pointed.
\end{lemma}

\begin{proof}
  Let \(F\colon J^{\triangleright}\to\Fun(K_0,\cat{C})\)
  be a colimit diagram where \(J\) is weakly contractible.
  We need to show that \(i_*(F(\X))(k)
  \colon J^{\triangleright}\to\cat{C}\)
  is colimiting for any \(k\in K\).
  If \(k\in K_0\), the diagram is equivalent to \((F(\X))(k)\),
  which is colimiting since so is~\(F\).
  If \(k\notin K_0\), 
  the diagram is equivalent
  to the constant diagram with value~\(*\),
  which is colimiting since \(J\) is weakly contractible.
\end{proof}

\begin{lemma}\label{b8ebfe9ce2}
  Let \(P\) be a poset
  and \(\cat{C}\) a compactly generated \(\infty\)-category.
  Then any compact object of \(\Fun(P,\cat{C})\)
  is a left Kan extension of
  its restriction to some finite full subposet.
  If \(P\) is finite,
  the full subcategory of compact objects
  is the essential image of
  the inclusion
  \(\Fun(P,\cat{C}^{\omega})\hookrightarrow\Fun(P,\cat{C})\).
\end{lemma}

\begin{proof}
  These follow from
  \autocite[Corollary~2.11
  and Proposition~2.8]{ttg-fun}\footnote{
    Beware that the assumption
    \(\bigcup_{j\in J}K_j=K\) is missing
    in the statement
    of \autocite[Corollary~2.11]{ttg-fun}
  }, respectively.
\end{proof}

\begin{lemma}\label{fa97259e33}
  Let \(P\) be a locally finite poset
  and \(\cat{C}\) a compactly generated pointed \(\infty\)-category.
  Then for any \(P_0\in\Down(P)\),
  the right Kan extension functor
  \(i_*\colon\Fun(P_0,\cat{C})\hookrightarrow\Fun(P,\cat{C})\)
  preserves compact objects.
\end{lemma}

\begin{proof}
  Let \(p\in P_0\) be an element
  and \(C\) a compact object of~\(\cat{C}\).
  Since \(i_*\) preserves (finite) colimits
  by \cref{0ec8cf3aa6},
  it suffices to show that
  \(F=i_*(j(p)\otimes C)\) is compact,
  where \(j\)
  denotes the Yoneda embedding \(P_0^{\op}\hookrightarrow\Fun(P_0,\Cat{S})\).
  Now we compute \(F(q)\) for \(q\in P\):
  If \(q\in P_{\geq p}\cap P_0\), it is \(C\).
  If \(q\notin P_0\), it is final
  and thus initial since \(\cat{C}\) is pointed.
  Otherwise, it is initial.
  This computation shows that \(F\rvert_{P\setminus P_{\geq p}}\) is initial,
  which means that
  \(F\) is the left Kan extension of \(F\rvert_{P_{\geq p}}\),
  as \(P_{\geq p}\) is upward closed.
  This computation also shows that
  \(F\rvert_{P_{\geq p}}\) takes compact values,
  which means by \cref{b8ebfe9ce2} that \(F\rvert_{P_{\geq p}}\) is compact,
  as \(P_{\geq p}\) is finite.
  Hence the desired result follows.
\end{proof}

\begin{lemma}\label{8eca80f4cb}
  Let \(P\) be a locally finite poset
  and \(\cat{C}\) a compactly generated pointed \(\infty\)-category.
  Then every compact object in \(\Fun(P,\cat{C})\)
  is a right Kan extension of
  its restriction to
  \(\bigcup_{s\in S}P_{\leq s}\)
  for some \(S\in\Pow^{\fin}(P)\).
\end{lemma}

\begin{proof}
  Let \(F\) be a compact object.
  By \cref{b8ebfe9ce2},
  there is a finite full subposet~\(Q\)
  such that \(F\) can be identified with the
  left Kan extension of \(F\lvert_Q\) along \(Q\hookrightarrow P\).
  We take \(S=\bigcup_{q\in Q}P_{\geq q}\),
  which is finite since \(P\) is locally finite,
  and consider the inclusion
  \(i\colon P_S=\bigcup_{s\in S}P_{\leq s}\hookrightarrow P\).
  Since \(P_S\) contains~\(Q\), the morphism
  \(i_!i^*F\to F\) is an equivalence.
  Hence it suffices to show that
  the composite \(i_!i^*F\to F\to i_*i^*F\) is an equivalence.
  As its restriction to~\(P_S\) is an equivalence,
  we consider \(p\notin P_S\).
  Then \((i_!i^*F)(p)\) is initial
  since no \(q\in Q\) satisfies \(q\leq p\)
  and \((i_*i^*F)(p)\) is final
  since \(P_S\) is downward closed.
  Since \(\cat{C}\) is pointed, the desired claim follows.
\end{proof}

\begin{proof}[Proof of \cref{13a6000b87}]
  According to \cref{fa97259e33},
  the diagram actually lands in \(\Cat{Pr}_{\cat{\omega}}\),
  the \(\infty\)-category of compactly generated \(\infty\)-categories
  and functors preserving colimits and compact objects.
  Since the inclusion
  \(\Cat{Pr}_{\omega}\hookrightarrow\Cat{Pr}\)
  preserves colimits
  by \autocite[Theorem~5.5.3.18 and Proposition~5.5.7.6]{LurieHTT},
  it suffices to show
  that its restriction \(\Pow^{\fin}(P)^{\triangleright}
  \to\Cat{Pr}_{\omega}\) is colimiting.
  Furthermore,
  since \(\Pow^{\fin}(P)\) is filtered,
  it suffices to show 
  that its composite with \((\X)^{\omega}
  \colon\Cat{Pr}_{\cat{\omega}}\to\Cat{Cat}_{\infty}\)
  is colimiting.
  Then the desired claim follows from \cref{cbb999bf61,8eca80f4cb}.
\end{proof}

One might think that
the desired equivalence could be immediately
obtained from \cref{8d61256fe2,13a6000b87}
by taking the colimit of
the assignment
given by
\begin{equation*}
  \phantom,
  \Pow^{\fin}(P)\ni
  S\longmapsto
  \bigl(\DD_{P_S}
  \colon\Fun(P_S,\Cat{Sp})
  \to
  \Fun(P_S^{\op},\Cat{Sp})
  \bigr)\in\Fun(\Delta^1,\Cat{Pr})
  ,
\end{equation*}
where \(P_S\) denotes \(\bigcup_{s\in S}P_{\leq s}\).
However, what we have proven in \cref{ss-pot} is
not sufficient in order to construct such a functor directly.
We avoid this issue by first constructing the desired functor~\(\DD\)
for~\(P\):

\begin{definition}\label{eb22338bf4}
  For a locally finite poset~\(P\),
  we define
  \(\Gamma_{\cpt}\colon\Fun(P,\Cat{Sp})\to\Cat{Sp}\)
  as the colimit of the functor
  \(\Pow^{\fin}(P)\to\Fun(\Delta^1,\Cat{Pr})\)
  given by \(Q\mapsto(\Gamma\colon\Fun(Q,\Cat{Sp})\to\Cat{Sp})\).
  Note that the source is identified with \(\Fun(P,\Cat{Sp})\)
  by \cref{13a6000b87}
  and the target is identified with \(\Cat{Sp}\)
  by \cref{cbb999bf61}
  and the fact that \(\Down^{\fin}(P)\) is weakly contractible.
  From this,
  we obtain
  \(\DD\colon\Fun(P,\Cat{Sp})\to\Fun(P^{\op},\Cat{Sp})\)
  as in \cref{ss-pot}.
\end{definition}

Now the following two results imply \cref{lf-arbitrary}:

\begin{proposition}\label{924159442f}
  Let \(P\)
  be a locally finite poset
  and
  \(F\colon P\to\Cat{Sp}\)
  a functor.
  If \(P_{\leq p}\)
  is finite for each \(p\in P\),
  the functor \(\DD(F)\colon P^{\op}\to\Cat{Sp}\)
  is pointwise given by
  \(p\mapsto\projlim_{q\in P}\Map(p,q)\otimes F(q)\).
\end{proposition}

\begin{proof}
  We fix~\(p\) and vary~\(F\).
  Then
  \(F\mapsto\projlim_{q\in P}\Map(p,q)\otimes F(q)\)
  preserves colimits
  by the finiteness assumption on~\(P\).
  Hence we can assume that \(F\) is compact.
  By \cref{8eca80f4cb},
  we can find \(S\in\Pow^{\fin}(P)\)
  such that \(F\)
  is a right Kan extension of
  its restriction to
  \(\bigcup_{s\in S}P_{\leq s}\).
  By replacing \(S\) with \(S\cup\{p\}\),
  we can assume \(p\in S\).
  Then the desired result
  follows from \cref{f6072744fd}
  since \(\bigcup_{s\in S}P_{\leq s}\)
  is finite by assumption.
\end{proof}

\begin{theorem}\label{7873d83c85}
  Let \(P\)
  be a locally finite poset.
  If \(P_{<p}\) is finite and Verdier
  for each \(p\in P\),
  the pair \((\Fun(P,\Cat{Sp}),\Gamma_{\cpt})\)
  is a commutative Frobenius algebra
  in \(\Cat{Pr}_{\st}\).
  In particular, \(\DD\) is an equivalence.
\end{theorem}

\begin{proof}
  By \cref{a3ad5064a0},
  we only need to show that \(\DD\) is an equivalence.
  For \(S\in\Pow^{\fin}(P)\),
  let \(P_S\) denote \(\bigcup_{s\in S}P_{\leq s}
  \in\Down(P)\).
  We regard \(\DD\) as an object
  of \(\Fun(\Delta^1,\Cat{Pr}_{\st})\)
  and consider the (essential) poset of
  subobjects \(\Sub(\DD)\).
  By \cref{00461cedb4}
  applied to \cref{50e6a4f108},
  each \(S\in\Pow^{\fin}(P)\) determines
  \(\DD_{P_S}\in\Sub(\DD)\).
  Hence we obtain
  the morphism of posets
  \(\Pow^{\fin}(P)\to\Sub(\DD)\).
  Then we consider
  the composite
  \begin{equation*}
    \phantom,
    \Pow^{\fin}(P)^{\triangleright}
    \longrightarrow\Sub(\DD)
    \longrightarrow\Fun(\Delta^1,\Cat{Pr}_{\st}),
  \end{equation*}
  where we set \(\infty\mapsto\DD\)
  in the first arrow.
  This is colimiting
  by \cref{13a6000b87,8d61256fe2}.
  By assumption
  and \cref{main},
  the functor
  \(\DD_{P_S}\) is an equivalence
  for \(S\in\Pow^{\fin}(P)\).
  Therefore, \(\DD\) is also an equivalence.
\end{proof}

\begin{remark}\label{53c4ee6c29}
  We can define the lower shriek functor
  for a morphism between posets
  satisfying the condition of \cref{lf-arbitrary}
  as the functor
  corresponding to the cosheaf pushforward
  under the duality equivalences.
  See \cref{0f134964dc}
  for the locally compact Hausdorff case.
\end{remark}

\subsection{In terms of sheaves}\label{ss-alex}

We explain that
our duality for a poset can be interpreted as 
a sheaf-cosheaf duality over its Alexandroff space,
which we recall as follows:

\begin{definition}\label{e407564b20}
  The \emph{Alexandroff space} \(\Alex(P)\)
  of a poset~\(P\)
  is the topological space
  whose underlying set is that of~\(P\) and
  whose open sets are the upward-closed subsets.
\end{definition}

We recall the following fact,
which was first proven in \autocite[Example~A.11]{ttg-fun}.

\begin{theorem}[Aoki]\label{7af5d021e3}
  The assignment \(F\mapsto(p\mapsto F(P_{\geq p}))\)
  determines the inverse image functor of a geometric morphism
  \begin{equation}
    \phantom.
    \label{e-5a26004a}
    \Fun(P,\Cat{S})=\PShv(P^{\op})\longrightarrow\Shv(\Alex(P)).
  \end{equation}
  This identifies \(\Shv(\Alex(P))\)
  as the bounded reflection of \(\PShv(P^{\op})\)
  and \(\PShv(P^{\op})\) as
  the hypercompletion of \(\Shv(\Alex(P))\).
\end{theorem}

Note that this geometric morphism is not an equivalence in general;
see \autocite[Example~A.13]{ttg-fun}.
However, it is an equivalence in the situation we are interested in:

\begin{proposition}\label{f5b86f5e98}
  If \(P\) is a locally finite poset,
  \cref{e-5a26004a} is an equivalence.
\end{proposition}

\begin{proof}
  According to \autocite[Example~A.12]{ttg-fun},
  this is true for finite posets.
  By \cref{7af5d021e3},
  the morphism \cref{e-5a26004a} is an equivalence
  if and only if \(\Shv(\Alex(P))\) is hypercomplete.
  Since \(\Shv(\Alex(P))\)
  can be written as a colimit of
  \(\Shv(\Alex(P_{\geq p_1}\cap\dotsb\cap P_{\geq p_n}))\)
  for \(p_1\), \dots,~\(p_n\in P\) and \(n\geq1\)
  in the \(\infty\)-category of 
  \(\infty\)-toposes,
  \(\Shv(\Alex(P))\) is hypercomplete
  when \(P\) is locally finite.
\end{proof}

\begin{remark}\label{c29c50ee4d}
  Note that by using \autocite[Corollary~2.6]{AsaiShah}
  instead of \autocite[Example~A.12]{ttg-fun} in the proof,
  we can obtain this result for a wider class of posets.
\end{remark}

\begin{remark}\label{c204a3567d}
  It is a consequence of
  \autocite[Theorem~3.4]{ClausenJansen}
  that the morphism \cref{e-5a26004a}
  is an equivalence for a poset
  satisfying the ascending chain condition.
  However, they use
  the ``geometric morphism''
  \(\Shv(X)\to\PShv(P^{\op})\)
  constructed in \autocite[page~27]{ClausenJansen}
  for a stratification \(X\to\Alex(P)\),
  which is not geometric in general;
  the trivial stratification on~\(\Alex(P)\)
  for the poset~\(P\) in \autocite[Example~A.13]{ttg-fun}
  gives a counterexample.
  Nevertheless,
  when \(P\) is locally finite,
  \cref{f5b86f5e98} shows that
  the morphism is indeed geometric.
\end{remark}

Hence \cref{lf-arbitrary},
which we have seen in \cref{ss-lf},
says the following:

\begin{theorem}\label{8d0ab00507}
  Let \(P\) be a locally finite poset
  such that \(P_{<p}\)
  is finite and Gorenstein* for each \(p\in P\).
  Then there is a canonical equivalence
  \begin{equation*}
    \phantom.
    \DD\colon
    \Shv_{\Cat{Sp}}(\Alex(P))
    \longrightarrow\cShv_{\Cat{Sp}}(\Alex(P)).
  \end{equation*}
\end{theorem}

\section{Verdier duality for proper separated \texorpdfstring{\(\infty\)}{\textinfty}-toposes}\label{s-ch}

The sheaf-cosheaf duality
for locally compact Hausdorff spaces,
which is often called covariant Verdier duality,
was studied in \autocite[Section~5.5.5]{LurieHA}.
In this section, we first prove its generalization
using more abstract methods.
Then we prove \cref{str} using our formulation.
In future work,
we will study a relative variant.

\subsection{Proper separated \texorpdfstring{\(\infty\)}{\textinfty}-toposes}\label{ss-ps}

Following
\autocite[C2.4.16]{Elephant},
we say that a geometric morphism
is \emph{Beck--Chevalley}
if any pullback satisfies
the Beck--Chevalley condition;
i.e., the (unstable) proper base change theorem holds.
Recall that
in \autocite[Section~7.3.1]{LurieHTT}
a geometric morphism is called \emph{proper}
if its arbitrary base change is Beck--Chevalley.

\begin{definition}\label{53e4be788e}
  An \(\infty\)-topos \(\cat{X}\) is called \emph{separated}
  if its diagonal \(\cat{X}\to\cat{X}\times\cat{X}\) is proper.
\end{definition}

\begin{remark}\label{ea7702d67f}
  Consider a geometric morphism \(\cat{Y}\to\cat{X}\)
  between \(n\)-toposes.
  If the geometric morphism
  \(\Shv(\cat{Y})\to\Shv(\cat{X})\) between \(\infty\)-toposes
  is proper,
  its arbitrary base change is Beck--Chevalley
  in the \((n+1)\)-category of \(n\)-toposes,
  but not vice versa.
  This is why in \(1\)-topos theory
  we usually call a geometric morphism \emph{tidy}
  when its arbitrary base change is Beck--Chevalley
  in the \(2\)-category of \(1\)-toposes.
  The same remark applies to the notion of separatedness.
\end{remark}

However,
the following is proven 
in \autocite[Theorem~7.3.1.16]{LurieHTT}:

\begin{example}[Lurie]\label{884b0546f5}
  The sheaf \(\infty\)-topos
  of a compact Hausdorff space is proper and separated.
\end{example}

We recall the following notion,
which was introduced in \autocite[Appendix~D]{Gaitsgory15}:

\begin{definition}[Gaitsgory]\label{95976481a6}
  A presentably symmetric monoidal stable \(\infty\)-category \(\cat{C}\)
  is called \emph{rigid}
  if the unit \(u\colon\Cat{Sp}\to\cat{C}\) admits
  a colimit-preserving right adjoint
  and the multiplication \(m\colon\cat{C}\otimes\cat{C}\to\cat{C}\) admits
  a \(\cat{C}\otimes\cat{C}\)-linear\footnote{
    Here the colimit-preserving property
    is included in the definition of linearity.
  }
  right adjoint.
\end{definition}

If \(\cat{C}\) is rigid,
it is easy to see that
\(u^{\R}\circ m\) and \(m^{\R}\circ u\)
constitute a duality datum in \(\Cat{Pr}\),
where \(\X^{\R}\) indicates the right adjoint.
In particular,
\((\cat{C},u^{\R}\circ m)\) is a commutative Frobenius algebra.

\begin{theorem}\label{2bd8a0da2f}
  If \(\cat{X}\) is a proper separated \(\infty\)-topos,
  then \(\Shv_{\Cat{Sp}}(\cat{X})\) is rigid.
\end{theorem}

\begin{corollary}\label{600698211c}
  The pair
  \((\Shv_{\Cat{Sp}}(\cat{X}),\Gamma)\)
  is a commutative Frobenius algebra in \(\Cat{Pr}_{\st}\)
  for any proper separated \(\infty\)-topos~\(\cat{X}\).
\end{corollary}

\begin{proof}[Proof of \cref{2bd8a0da2f}]
  According to \autocite[Example~4.8.1.19]{LurieHA},
  the binary product of \(\infty\)-toposes
  can be computed as their tensor product in \(\Cat{Pr}\).
  Hence the result follows from \cref{90680bda5b} below.
\end{proof}

\begin{lemma}\label{90680bda5b}
  Let \(f\colon\cat{Y}\to\cat{X}\) be a proper morphism
  of \(\infty\)-toposes,
  then \(f^*\colon
  \Shv_{\Cat{Sp}}(\cat{X})\to\Shv_{\Cat{Sp}}(\cat{Y})\)
  admits a \(\Shv_{\Cat{Sp}}(\cat{X})\)-linear right adjoint.
\end{lemma}

\begin{proof}
  According to \autocite[Remark~7.3.1.5]{LurieHTT},
  the direct image functor
  \(\cat{Y}\to\cat{X}\) preserves filtered colimits.
  Hence \(f_*\colon\Shv_{\Cat{Sp}}(\cat{Y})\to\Shv_{\Cat{Sp}}(\cat{Y})\)
  preserves colimits.
  Now we consider the diagram
  \begin{equation*}
    \begin{tikzcd}[column sep=large]
      \cat{Y}\ar[r,"\text{graph}"]\ar[d,"f"']&
      \cat{Y}\times\cat{X}\ar[r,"\pr_1"]\ar[d,"f\times{\id}"]&
      \cat{Y}\ar[d,"f"]\\
      \cat{X}\ar[r,"\text{diagonal}"]&
      \cat{X}\times\cat{X}\ar[r,"\pr_1"]&
      \cat{X}
    \end{tikzcd}
  \end{equation*}
  in the \(\infty\)-category of \(\infty\)-toposes.
  Since the right and outer squares are cartesian,
  so is the left one.
  According to \autocite[Example~4.8.1.19]{LurieHA},
  the binary product of \(\infty\)-toposes
  can be computed as their tensor product in \(\Cat{Pr}\).
  Therefore,
  since \(f\times{\id}\)
  is Beck--Chevalley,
  for any \(F\in\Shv_{\Cat{Sp}}(\cat{X})\)
  and \(G\in\Shv_{\Cat{Sp}}(\cat{Y})\)
  the canonical morphism
  \(f_*G\otimes F\to f_*(G\otimes f^*F)\) is an equivalence.
\end{proof}

\subsection{The locally compact case}\label{ss-local}

The following
result is
derived from \cref{2bd8a0da2f}
by using \cref{ee41f58df5}
applied to \cref{f7ea8c5d40}:

\begin{theorem}\label{bb6fe1df93}
  Let \(j\colon\cat{U}\hookrightarrow\cat{X}\)
  be an open subtopos
  of a proper separated \(\infty\)-topos.
  Then the pair \((\Shv_{\Cat{Sp}}(\cat{U}),\Gamma_{\cat{X}}\circ j_!)\)
  is a commutative Frobenius algebra
  in \(\Cat{Pr}_{\st}\).
\end{theorem}

Here the composite \(\Gamma_{\cat{X}}\circ j_!\)
depends on~\(j\),
not only on~\(\cat{U}\),
but there is a canonical choice
for locally compact spaces:

\begin{definition}\label{02c27cb552}
  Let \(X\) be a locally compact Hausdorff space.
  We define
  the \emph{global section with compact support}
  \(\Gamma_{\cpt}\) as the composite \(p_*\circ j_!\)
  where \(j\colon X\hookrightarrow X_{\infty}\)
  is the inclusion to its one-point compactification
  and \(p\colon X_{\infty}\to{*}\) is the projection.
  Then \cref{bb6fe1df93} says that
  \((\Shv_{\Cat{Sp}}(X),\Gamma_{\cpt})\)
  is a commutative Frobenius algebra.
  We let \(\DD\colon\Shv(X)\to\cShv(X)\)
  denote the associated equivalence
  (cf. \cref{ss-pot}).
\end{definition}

\begin{remark}\label{12b6445923}
  One can prove Verdier duality for
  locally compact Hausdorff spaces
  by a similar method
  to the one we have used in \cref{ss-lf}
  for locally finite posets:
  Namely, the sheaf and cosheaf \(\infty\)-categories
  of a locally compact Hausdorff space
  can be written as colimits in \(\Cat{Pr}\)
  as those of compact subspaces.
  We leave the details to the interested reader.
\end{remark}

We give its objectwise description
to justify calling our functor ``Verdier duality'':

\begin{proposition}\label{bdd139fe5d}
  For a locally compact Hausdorff space~\(X\)
  and a spectrum-valued sheaf \(F\in\Shv_{\Cat{Sp}}(X)\),
  the cosheaf \(\DD(F)\)
  is pointwise given by
  \(U\mapsto
  \injlim_{K\subset U}\fib(F(X)\to F(X\setminus K))\),
  where \(K\) runs over compact subsets.
\end{proposition}

Note that
Lurie's equivalence also has this pointwise formula;
see \autocite[Proposition~5.5.5.10]{LurieHA}.

\begin{proof}
  First suppose that \(X\) is compact.
  Let \(j\) denote the inclusion
  \(U\hookrightarrow X\).
  By definition,
  \(\DD(F)(U)\) is the global section
  of \((j_!\SS_U)\otimes F\).
  Let \(i\) denote the inclusion \(X\setminus U\hookrightarrow X\).
  Then by recollement,
  \(\DD(F)(U)\)
  is equivalent to the global section of \(\fib(F\to i_*i^*F)\).
  Hence it is written as
  \(\injlim_{V\supset X\setminus U}
  \fib(F(X)\to F(V))\),
  where \(V\) runs over open subsets.
  As \(X\) is compact, this
  coincides with the desired description.

  We proceed to the general case.
  Let \(j\colon X\hookrightarrow X_{\infty}\)
  denote the inclusion to
  the one-point compactification
  and \(i\) the inclusion of the point at infinity.
  \Cref{00461cedb4} applied to \cref{f7ea8c5d40}
  says \(j_+\circ\DD_X\simeq\DD_{X_{\infty}}\circ j_!\).
  Hence \(\DD_X(F)(U)\) can be computed as
  \begin{equation*}
    \phantom,
    (j_+\circ\DD_X)(F)(U)
    \simeq(\DD_{X_{\infty}}\circ j_!)(F)(U)
    \simeq\injlim_{K\subset U}
    \fib\bigl((j_!F)(X_{\infty})\to(j_!F)(X_{\infty}\setminus K)\bigr),
  \end{equation*}
  where we use the compact case.
  By recollement,
  the desired result follows from
  the vanishing of
  \(\fib((i_*i^*j_*F)(X_{\infty})
  \to(i_*i^*j_*F)(X_{\infty}\setminus K))\)
  for each \(K\),
  which follows from \(K\subset X\).
\end{proof}

\begin{remark}\label{0f134964dc}
  Let \(f\colon Y\to X\) be a continuous map
  between locally compact Hausdorff spaces.
  As in
  \autocite[Remark~9.4.6]{GaitsgoryLurie},
  we can define the lower shriek functor~\(f_!\)
  as the composite \((\DD_X)^{-1}\circ f_+\circ\DD_Y\).
  One could
  check its standard properties by
  applying \cref{00461cedb4} to \cref{68011c91d6,f7ea8c5d40}.
  To describe further
  functorial properties
  of this construction, one could use the technology presented
  in \autocite[Chapter~7]{GaitsgoryRozenblyum171}.
  However, beware that it is built on unproven results
  in \((\infty,2)\)-category theory.
\end{remark}

\subsection{Application: Verdier duality and stratification}\label{ss-str}

We prove the following generalization of \cref{str}:

\begin{theorem}\label{a292d59cdb}
  Let \(P\) be a finite poset
  and \(\cat{X}\to\Shv(\Alex(P))\)
  a geometric morphism.
  Suppose that \(P\) is Verdier,
  that \(\cat{X}\) is proper and separated,
  and that
  the spectrum-valued inverse image
  \(f^*\colon\Shv_{\Cat{Sp}}(\Alex(P))\to\Shv_{\Cat{Sp}}(X)\)
  is fully faithful.
  Then we have \(\DD_P\simeq f_+\circ\DD_X\circ f^*\).
\end{theorem}

\begin{remark}\label{7014e1bd6e}
  The assumption is satisfied
  when the \emph{space-valued} inverse image
  \(\Shv(\Alex(P))\to\cat{X}\)
  is fully faithful:
  This can be seen
  by considering
  objects of \(\Shv_{\Cat{Sp}}(\X)\)
  as
  left exact functors \((\Cat{Sp}^{\omega})^{\op}\to\X\).
\end{remark}

\begin{proof}[Proof of \cref{a292d59cdb}]
  We have
  \(\Gamma_{\Alex(P)}
  \simeq\Gamma_{\Alex(P)}\circ f_*\circ f^*
  \simeq\Gamma_{X}\circ f^*\).
  Hence the desired result follows from \cref{e894284df5}.
\end{proof}

\bibliographystyle{plain}
\bibliography{references.bib}

\begin{thebibliography}{10}

\bibitem{ttg-fun}
Ko~Aoki.
\newblock Tensor triangular geometry of filtered objects and sheaves, 2020.
\newblock \href{https://arxiv.org/abs/2001.00319v1}{\tt arXiv:2001.00319v1}.

\bibitem{AsaiShah}
Ryo Asai and Jay Shah.
\newblock Algorithmic canonical stratifications of simplicial complexes, 2022.
\newblock \href{https://arxiv.org/abs/1808.06568v3}{\tt arXiv:1808.06568v3}.

\bibitem{Baclawski80}
Kenneth Baclawski.
\newblock Cohen--{{Macaulay}} ordered sets.
\newblock {\em Journal of Algebra}, 63(1):226--258, 1980.

\bibitem{Bjorner84}
A.~Bj{\"o}rner.
\newblock Posets, regular {{CW}} complexes and {{Bruhat}} order.
\newblock {\em European Journal of Combinatorics}, 5(1):7--16, 1984.

\bibitem{ClausenJansen}
Dustin Clausen and Mikala~{\O}rsnes Jansen.
\newblock The reductive {{Borel--Serre}} compactification as a model for
  unstable algebraic {{K}}-theory, 2021.
\newblock \href{https://arxiv.org/abs/2108.01924v1}{\tt arXiv:2108.01924v1}.

\bibitem{Curry18}
Justin~Michael Curry.
\newblock Dualities between cellular sheaves and cosheaves.
\newblock {\em Journal of Pure and Applied Algebra}, 222(4):966--993, 2018.

\bibitem{Gaitsgory15}
Dennis Gaitsgory.
\newblock Sheaves of categories and the notion of $1$-affineness.
\newblock In {\em Stacks and Categories in Geometry, Topology, and Algebra},
  volume 643 of {\em Contemp. {{Math}}.}, pages 127--225. {Amer. Math. Soc.,
  Providence, RI}, 2015.

\bibitem{GaitsgoryLurie}
Dennis Gaitsgory and Jacob Lurie.
\newblock Weil's conjecture for function fields.
\newblock available at the second author's website, 2014.

\bibitem{GaitsgoryRozenblyum171}
Dennis Gaitsgory and Nick Rozenblyum.
\newblock {\em A Study in Derived Algebraic Geometry. {{Volume I}}:
  {{Correspondences}} and Duality}, volume 221.
\newblock {American Mathematical Society (AMS)}, {Providence, RI}, 2017.

\bibitem{Elephant}
Peter~T. Johnstone.
\newblock {\em Sketches of an Elephant: A Topos Theory Compendium}.
\newblock Oxford {{Logic Guides}}. {The Clarendon Press, Oxford University
  Press, New York}, 2002.

\bibitem{LurieHTT}
Jacob Lurie.
\newblock {\em Higher topos theory}, volume 170 of {\em Annals of Mathematics
  Studies}.
\newblock Princeton University Press, Princeton, NJ, 2009.

\bibitem{LurieHA}
Jacob Lurie.
\newblock Higher algebra.
\newblock available at the author's website, 2017.

\bibitem{Schneider98}
Peter Schneider.
\newblock Verdier duality on the building.
\newblock {\em Journal f{\"u}r die Reine und Angewandte Mathematik},
  494:205--218, 1998.

\bibitem{Stanley96}
Richard~P. Stanley.
\newblock {\em Combinatorics and commutative algebra}, volume~41 of {\em
  Progress in Mathematics}.
\newblock Birkh\"auser Boston, Inc., Boston, MA, second edition, 1996.

\end{thebibliography}

\end{document}